\documentclass[reqno]{amsart}
\usepackage{amssymb,latexsym}
\usepackage{enumerate}
\usepackage{graphicx}
\usepackage{amsmath}
\usepackage{amsfonts}

\newtheorem{thm}{Theorem}[section]
\newtheorem{cor}[thm]{Corollary}
\newtheorem{lem}[thm]{Lemma}

\newtheorem{prop}[thm]{Proposition}

\theoremstyle{definition}

\numberwithin{equation}{section}
\begin{document}
\title[A note on Reidemeister Torsion of G-Anosov Representations]{A note on Reidemeister Torsion of G-Anosov Representations}

\author[H. Zeybek]{Hat\.{I}ce ZEYBEK}
\address{Hacettepe University, Department of Mathematics,
06800 Ankara, Turkey} \email{haticezeybek@hacettepe.edu.tr}

\author[Y. S\"{o}zen]{Ya\c{s}ar S\"{o}zen}
\address{Hacettepe University, Department of Mathematics,
06800 Ankara, Turkey} \email{ysozen@hacettepe.edu.tr}



\begin{abstract}
This article considers $G$-Anosov representations of a fixed
closed oriented Riemann surface $\Sigma$ of genus at least $2$.
Here, $G$ is the Lie group $\text{PSp}(2n,\mathbb{R}$), $\text{PSO}(n,n)$ or
$\text{PSO}(n,n+1)$. It proves that Reidemeister torsion (R-torsion)
associated to  $\Sigma$ with coefficients in the adjoint bundle
representations of such representations is well-defined. Moreover,
by using symplectic chain complex method, it establishes a novel
formula for R-torsion of such representations in terms of the
Atiyah-Bott-Goldman symplectic form corresponding to the Lie group
$G$. Furtermore, it applies the results to Hitchin components, in
particular, Teichm\"{u}ller space.
\end{abstract}

\maketitle

\section{Introduction}
Throughout the paper, $\Sigma$ is a closed Riemann surface of genus
$g\geq 2.$ \emph{Teichm\"{u}ller space} $\mathrm{Teich}(\Sigma)$ of
$\Sigma$ is the space of isotopy classes of complex structures on
$\Sigma.$ It is a differentiable manifold and diffeomorphic to a
ball of dimension $6g-6.$ By the Uniformization Theorem, it can be
interpreted as the isotopy classes of hyperbolic metrics on
$\Sigma,$ i.e. Riemannian metrics of constant Gaussian curvature
$(-1).$ It can also be interpreted as
$\mathrm{Hom}_{\mathrm{df}}(\pi_1(\Sigma),\mathrm{PSL}(2,\mathbb{R}))$
discrete, faithful representations of the fundamental group
$\pi_1(\Sigma)$ of $\Sigma$ to $\mathrm{PSL}(2,\mathbb{R}).$ This is
a connected component of the space
$\mathrm{Rep}(\pi_1(\Sigma),\mathrm{PSL}(2,\mathbb{R}))=\mathrm{Hom}^{+}(\pi_1(\Sigma),\mathrm{PSL}(2,\mathbb{R}))/\mathrm{PSL}(2,\mathbb{R})$
of all reductive representations of $\pi_1(\Sigma)$  to
$\mathrm{PSL}(2,\mathbb{R}).$

In the paper \cite{Hitchin}, N. Hitchin proved the existence of an
analogous component of $\mathrm{Rep}(\pi_1(\Sigma),G),$ where $G$
is a split real semi-simple Lie group, such as
$\mathrm{PSL}(n,\mathbb{R}),$ $\mathrm{PSp}(2n,\mathbb{R}),$
$\mathrm{PO}(n,n+1),$ and $\mathrm{PO}(n,n).$ He called this
component \emph{Teichm\"{u}ller component} but now it is called
\emph{Hitchin component}. He proved that Hitchin component of
$\mathrm{Hom}(\pi_1(\Sigma),G)/G$ is diffeomorphic to
$\mathbb{R}^{(6g-6)\dim G}.$ He also paused the problem about the
geometric significance of this component.

We already mentioned the geometric significance of the Hitchin
component for $G=\mathrm{PSL}(2,\mathbb{R}).$ Namely, the hyperbolic
structures on $\Sigma.$ For $G=\mathrm{PSL}(3,\mathbb{R}),$ S. Choi
and W.M. Goldman proved that the Hitchin component is diffeomorphic
to convex real projective structures on $\Sigma$ \cite{ChoiGoldman}. F. Labourie introduced the notion of Anosov representations in his 
investigation of Hitchin component by dynamical system method \cite{Labourie2006}, where he also proved that such representations are 
purely loxodromic, discrete, faithful, and irreducible.

The problem of giving a geometric interpretation of Hitchin
components was completely solved by O. Guichard and A. Wienhard
\cite{GuichardWienhard2}. To be more precise, they proved that the
Hitchin component of $\mathrm{Hom}(\pi_1(\Sigma ),G)/G$
parametrizes the deformation space of $(G,X)-$structures on a
compact manifold $M.$ Here, if $X = \mathbb{RP}^{2n-1}$ then G denotes $\mathrm{PSL}(2n,\mathbb{R}),\mathrm{PSp}(2n,\mathbb{R}) $ when $n\geq
2$, $\mathrm{PSO}(n,n)$ when $
n\geq 3$  or if $X =
\mathcal{F}_{1,2n}(\mathbb{R}^{2n+1})=\{(D,H)\in
\mathbb{RP}^{2n}\times \left(\mathbb{RP}^{2n}\right)^{\ast};
D\subset H\}$ then G denotes 
$\mathrm{PSL}(2n+1,\mathbb{R})$ when $n \geq 1,$  $\mathrm{PSO}(n,
n + 1)$ when $ n\geq 2$   \cite{GuichardWienhard2}. For details and more
information, we refer the reader to
\cite{GuichardWienhard,GuichardWienhard1,GuichardWienhard2}.


In the present paper, we consider $G$-Anosov representations where
the Lie group $G$ belongs to $\{\text{PSp}(2n,\mathbb{R}$), $\text{PSO}(n,n),\text{PSO}(n,n+1)\}$. We
showed that the Reidemeister torsion of such representations is
well-defined (Proposition
\ref{Well-definitenessOfRiedemeisterTorsionOfRepresentation}).
Furthermore, with the help of symplectic chain complex method, we
establish a novel formula for R-torsion of such representation in
terms of the Atiyah-Bott-Goldman symplectic form corresponding to
the Lie group $G$ (Theorem \ref{MainThm}).

\section{The Reidemeister Torsion}
 For more information and the detailed proofs, we refer the
reader to \cite{Porti,SozOJM,Tur2001,Tur2002,Witten}, and the
references therein.

Suppose $C_{\ast }=(C_n\stackrel{\partial_n}{\rightarrow}
C_{n-1}\rightarrow\cdots\rightarrow C_1
\stackrel{\partial_1}{\rightarrow} C_0\rightarrow 0)$ is a chain
complex  of a finite dimensional vector spaces over the field
$\mathbb{R}$ of real numbers. Let $H_p(C_\ast)$
$=Z_p(C_\ast)/B_p(C_\ast)$ denote the $p-$th homology group of
$C_\ast,$ $p=0,\ldots,n,$ where $B_p(C_\ast)=\mathrm{Im}\partial_{p+1}$ and $Z_p(C_\ast)=\mathrm{Ker}
\partial_p.$

Assume that $\mathbf{c}_p,$ $\mathbf{b}_p,$ and $\mathbf{h}_p$ are
bases of $C_p,$ $B_p(C_\ast),$  and $H_p(C_\ast),$ respectively, and
that $\ell_p:H_p(C_\ast)\to Z_p(C_\ast),$ $s_p:B_{p-1}(C_\ast)\to
C_p$ are sections of $Z_p(C_\ast)\to H_p(C_\ast),$ $C_p \to
B_{p-1}(C_\ast),$ respectively, $p=0,\ldots,n.$ The definition of
$Z_p(C_\ast), B_p(C_\ast),$ and $H_p(C_\ast)$ result the following
short-exact sequences: \begin{equation}
0 \rightarrow Z_p(C_\ast) \hookrightarrow C_p \twoheadrightarrow B_{p-1}(C_\ast) \rightarrow 0 ,
\end{equation}
\begin{equation}
0 \rightarrow B_p(C_\ast) \hookrightarrow Z_p(C_\ast) \twoheadrightarrow H_p(C_\ast) \rightarrow 0.
\end{equation} These short-exact sequences yield  a  new basis
$\mathbf{b}_p\sqcup \ell_p(\mathbf{h}_p)\sqcup
s_p(\mathbf{b}_{p-1})$ of $C_p.$

The \emph{Reidemeister torsion} of $C_{\ast}$ with respect to bases
$\{\mathbf{c}_p\}_{p=0}^n,$ $\{\mathbf{h}_p\}_{p=0}^{n}$ is defined
by $$\mathbb{T}\left(
C_{\ast},\{\mathbf{c}_p\}_{0}^n,\{\mathbf{h}_p\}_{0}^n \right)
=\prod_{p=0}^n \left[\mathbf{b}_p\sqcup \ell_p(\mathbf{h}_p)\sqcup
s_p(\mathbf{b}_{p-1}), \mathbf{c}_p\right]^{(-1)^{(p+1)}},$$ where
$\left[\mathbf{e}_p, \mathbf{f}_p\right]$ is the determinant of the
change-base-matrix from   $\mathbf{f}_p$ to $\mathbf{e}_p.$

The Reidemeister torsion
$\mathbb{T}\left(C_{\ast},\{\mathbf{c}_p\}_{0}^n,\{\mathbf{h}_p\}_{0}^n\right)$
is independent of the bases $\mathbf{b}_p,$ sections $s_p,\ell_p$
\cite{Milnor}. If $\mathbf{c}'_p,\mathbf{h}'_p$ are also bases
respectively for $C_p,$ $H_p(C_\ast),$ then  an easy computation
results the following change-base-formula:
\begin{equation}\label{change-base-formula}
 \mathbb{T}(C_{\ast},\{\mathbf{c}'_p\}_{0}^n,\{\mathbf{h}'_p\}_{0}^n)=
\displaystyle\prod_{p=0}^n\left(\dfrac{[\mathbf{c}'_p,\mathbf{c}_p]}
{[\mathbf{h}'_p,\mathbf{h}_p]}\right)^{(-1)^p}
 \mathbb{T}(C_{\ast},\{\mathbf{c}_p\}_{0}^n,\{\mathbf{h}_p\}_{0}^n).
\end{equation}
If
\begin{equation}\label{y1}
 0\to A_{\ast }\stackrel{\imath}{\to}
B_{\ast }\stackrel{\pi}{\to} D_{\ast }\to 0
\end{equation}
is a short-exact sequence of chain complexes, then we have the
long-exact sequence of vector spaces of length $3n+2$
\begin{equation}\label{y2}
\mathcal{H}_{\ast}:\; \cdots \to H_p(A_{\ast
})\stackrel{\imath_{p}}{\to } H_p(B_{\ast })\stackrel{\pi_{p}}{\to }
H_p(D_{\ast })\stackrel{\delta_{p}}{\to} H_{p-1}(A_{\ast
})\to\cdots,
\end{equation}
where $\mathcal{H}_{3p}=H_p(D_{\ast}),$
$\mathcal{H}_{3p+1}=H_p(A_{\ast}),$ and
$\mathcal{H}_{3p+2}=H_p(B_{\ast}).$

The bases $\mathbf{h}_p^D,$ $\mathbf{h}_p^A,$ and $\mathbf{h}_p^B$
are clearly  bases for $\mathcal{H}_{3p},$ $\mathcal{H}_{3p+1},$ and
$\mathcal{H}_{3p+2},$ respectively. Considering  sequences
(\ref{y1}) and (\ref{y2}), we have the following result of J.
Milnor:
\begin{thm}(\cite{Milnor})\label{MilA} Let
$\mathbf{c}^{A}_p,$ $\mathbf{c}^B_p,$  $\mathbf{c}^D_p,$
$\mathbf{h}^A_p,$ $\mathbf{h}^B_p,$ and $\mathbf{h}^D_p$ be bases of
$A_p,$ $B_p,$ $D_p,$  $H_p(A_{\ast }),$ $H_p(B_{\ast }),$ and
$H_p(D_{\ast }),$ respectively. Let $\mathbf{c}^A_p,$
$\mathbf{c}^B_p,$ and $\mathbf{c}^D_p$ be compatible in the sense
that $[\mathbf{c}^B_p,\mathbf{c}^A_p\oplus
\widetilde{\mathbf{c}^D_p}]=\pm 1,$ where
$\pi\left(\widetilde{\mathbf{c}^D_p}\right)=\mathbf{c}^D_p.$ Then,
\begin{align*}
\mathbb{T}(B_{\ast},\{\mathbf{c}^B_p\}_{0}^n,\{\mathbf{h}^B_p\}_{0}^n)&=
\mathbb{T}(A_{\ast},\{\mathbf{c}^A_p\}_{0}^n,\{\mathbf{h}^A_p\}_{0}^n)
\mathbb{T}(D_{\ast},\{\mathbf{c}^D_p\}_{p=0}^n,\{\mathbf{h}^D_p\}_{0}^n)\\
&\quad\times
\mathbb{T}(\mathcal{H}_{\ast},\{\mathbf{c}_{3p}\}_{0}^{3n+2},\{0\}_{0}^{3n+2}).
\end{align*}
\hfill $\Box$
\end{thm}
Considering the short exact sequence
\begin{equation*}
 0\to A_{\ast }\stackrel{\imath}{\to}
A_{\ast }\oplus D_{\ast }\stackrel{\pi}{\to} D_{\ast }\to 0,
\end{equation*}
where for $p=0,\ldots,n,$ $\imath_p:A_p\to A_p\oplus D_p$ denotes
the inclusion, $\pi_p:A_p\oplus D_p\to D_p$ denotes the projection,
and the compatible bases $\mathbf{c}^A_p,$ $\mathbf{c}^A_p\oplus
\mathbf{c}^D_p,$ and $\mathbf{c}^D_p,$  where we consider the
inclusion as a section of $\pi_p:A_p\oplus D_p\to D_p,$  then by
Theorem~\ref{MilA} we get:
\begin{lem}(\cite{SozMathScandinavia})\label{sumlemma}
If $A_{\ast},$ $D_{\ast}$ are two chain complexes, $\mathbf{c}^A_p,$
$\mathbf{c}^D_p,$ $\mathbf{h}^A_p,$ and $\mathbf{h}^D_p$ are bases
of $A_p,$ $D_p,$ $H_p(A_{\ast}),$ and $H_p(D_{\ast}),$ respectively,
then $$ \mathbb{T}(A_{\ast}\oplus D_{\ast},\{ \mathbf{c}^A_p\oplus
\mathbf{c}^D_p\}_0^n,\{ \mathbf{h}^A_p\oplus
\mathbf{h}^D_p\}_0^n)=\mathbb{T}(A_{\ast},\{ \mathbf{c}^A_p\}_0^n,\{
\mathbf{h}^A_p\}_0^n)\mathbb{T}(D_{\ast},\{\mathbf{c}^D_p\}_0^n,\{
\mathbf{h}^D_p\}_0^n).
$$
\hfill $\Box$
\end{lem}

Note that one can split a general chain complex as a
direct sum of an exact and a $\partial-$zero chain complexes.
Moreover, Reidemeister torsion $\mathbb{T}(C_{\ast})$ of a general
complex $C_{\ast}$ is as an element of $\otimes_{p=0}^n
(\det(H_p(C_{\ast})))^{(-1)^{p+1}},$ where
$\det(H_p(C_{\ast}))=\bigwedge^{\dim_{\mathbb{R}} H_p(C_{\ast})
}H_p(C_{\ast})$ denotes the top exterior power of $H_p(C_{\ast}),$
and $\det(H_p(C_{\ast}))^{-1}$ is the dual of $\det(H_p(C_{\ast})).$
We refer the reader \cite{SozOJM,Witten} for more information and
the detailed proofs.

A \emph{symplectic chain complex} is a chain complex of finite
dimensional real vector spaces $C_{\ast}:0\to
C_{2n}\stackrel{\partial_{2n}}{\to} C_{2n-1}\to\cdots\to
C_{n}\to\cdots \to C_1\stackrel{\partial_1}{\to} C_0\to 0$ of length
$2n$($n$ odd) together with for each $p=0,\ldots,n,$ a
$\partial-$compatible anti-symmetric non-degenerate bilinear form
$\omega_{p,2n-p}:C_p\times C_{2n-p}\to\mathbb{R}.$ Namely,
$$\omega_{p,2n-p}\left( \partial a,b\right)
=(-1)^{p+1}\omega_{p+1,2n-(p+1)}(a,\partial b),$$
$$
\omega_{p,2n-p}(a,b)=(-1)^{p}\omega_{2n-p,p}(b,a).$$

Clearly, we have anti-symmetric and non-degenerate bilinear map
$[\omega_{p,n-p}]:H_p(C_{\ast})\times H_{n-p}(C_{\ast})\to
\mathbb{R}$ defined by
$[\omega_{p,n-p}]([x],[y])=\omega_{p,n-p}(x,y).$

Suppose  $C_{\ast}$ is a real-symplectic chain complex of length
$2n,$ and that $\mathbf{c}_p$ is a basis of $C_p,$ $p=0,\ldots,2n.$
We say that the bases $\mathbf{c}_p$  of $C_p$ and
$\mathbf{c}_{2n-p}$ of $C_{2n-p}$ are $\omega-$\emph{compatible}, if
the matrix of $\omega_{p,2n-p}$ in bases $\mathbf{c}_p,$
$\mathbf{c}_{2n-p}$ is the $k\times k$ identity matrix
$\mathrm{I}_{k\times k}$ when $p\ne n$ and $\left(
\begin{array}{cc}
  0_{l\times l} & \mathrm{I}_{l\times l} \\
  -\mathrm{I}_{l\times l} & 0_{l\times l} \\
\end{array}
\right) $ when $p=n,$ where
$k=\dim_{\mathbb{R}}C_p=\dim_{\mathbb{R}} C_{2n-p}$ and
$2l=\dim_{\mathbb{R}} C_{n}.$

Let us introduce the following notation used throughout the paper.
Let $C_{\ast}$ be a real-symplectic chain complex and let
$\mathbf{h}^C_p,$ $\mathbf{h}^C_{2n-p}$ be bases of $H_p(C_{\ast}),$
$H_{2n-p}(C_{\ast}),$ respectively. Then,
$\Delta_{p,2n-p}(C_{\ast})$ denotes the determinant of the matrix of
the non-degenerate pairing $[\omega_{p,2n-p}]:H_p(C_{\ast})\times
H_{2n-p}(C_{\ast})\to \mathbb{R}$ in bases $\mathbf{h}^C_p,$
$\mathbf{h}^C_{2n-p}.$

By an easy linear algebra argument, we have:
\begin{lem}
(\cite{SozFundMathematicae})\label{EverySymplecticChainComplexHasOmegaCompatibleBases}
For a symplectic chain complex, there exist $\omega-$compatible
bases.

\hfill $\Box$
\end{lem}
\begin{thm}
(\cite{SozOJM})\label{Tor_Of_General_Symplectic} Suppose $C_{\ast}$
is a symplectic chain complex of length $2n,$ and  $\mathbf{c}_p,$
$\mathbf{h}^C_p$ are bases of $C_p,$ $H_p(C_{\ast}),$ respectively,
$p=0,\ldots,2n.$ Then, the following formula is valid:
$$
\mathbb{T}\left(
C_{\ast},\{\mathbf{c}_p\}_{p=0}^{2n},\{\mathbf{h}^C_p\}_{p=0}^{2n}\right)=
\prod_{p=0}^{n-1} \Delta_{p,2n-p}(C_{\ast})^{(-1)^p}\cdot
   \sqrt{\Delta_{n,n}(C_{\ast })} ^{\;(-1)^{n}}. $$
\hfill $\Box$
\end{thm}

We refer the reader to \cite{SozOJM} for detailed proof and
unexplained subjects. See also
\cite{SozTopAndApp,SozMS,SozMathScandinavia,SozFundMathematicae} for
further applications of Theorem~\ref{Tor_Of_General_Symplectic}.

\section{Anosov Representations}

Let $\Sigma$ be a closed oriented Riemann surface of genus at
least $2$, $h$ be a hyperbolic metric on $\Sigma$, $M=UT(\Sigma)$
be its unit tangent bundle of $\Sigma$ and $g_t$ be the geodesic flow
for the hyperbolic metric $h$. Since the geodesic flow on the unit
tangent bundle of a negatively curved manifold is
Anosov \cite{Katok}, $g_t$ is an Anosov flow. To be more precise, there 
is a $g_t$-invariant splitting of the tangent bundle
\[T\Sigma=E^s \oplus E^u \oplus E^t. \]
Here,

$\bullet$ $E^t$ is a line bundle, which is tangent to the flow
$g_t$,

$\bullet$ $E^u$ is expanding, namely, there are constant
$A>0,\alpha >1$ such that for each $t\in \mathbb{R}$ and $v\in
E^u$,
\[\|D_{g_t}(v)\|\geq A \alpha^t\|v\|,\]

$\bullet$ $E^s$ is contracting, in other words, there are
constants $B>0,0\leq \beta < 1$ such that for each $t \in
\mathbb{R} $ and $v \in E^s$,
\[\|D_{g_t}(v)\|\leq B \beta^t\|v\|.\]

Let $\widetilde{\Sigma}$ be the universal covering of $\Sigma$,
$\widehat{M}=UT(\widetilde{\Sigma})$ be the $\pi_1(\Sigma)$-cover
of $M$ and let us also denote by $g_t$ the geodesic flow on
$\widehat{M}$.

For a semi-simple Lie group G, let $(P^+,P^-)$ be a pair of
opposite parabolic subgroups of G. We denote respectively the
quotient spaces $G/P^+,G/P^-$, and $G/L$ by
$\mathcal{F}^+,\mathcal{F}^-,$ and $\mathcal{X}$, where $L=P^+
\cap P^-$. Note that considering the diagonal action of $G$ on
$\mathcal{F}^+\times \mathcal{F}^-,\mathcal{X}$ is the unique open
$G$-orbit. This product structure induces two G-invariant
distributions $E^+$ and $E^-$ on $\mathcal{X}$. To be more
precise, $E_x^+=T_{x_+}\mathcal{F}^+$ and
$E_x^-=T_{x_-}\mathcal{F}^-$, where $x=(x_+,x_-)\in
\mathcal{X}\subset \mathcal{F}^+\times \mathcal{F}^-$. Clearly,
any $\mathcal{X}$-bundle can be equipped with two distributions
which will be denoted by the same letters $E^+$ and $E^-$.

Let $\varrho:\pi_1(\Sigma)\rightarrow G$ be a representation. By the
diagonal action of $\pi_1(\Sigma),\widehat{M}\times \mathcal{X}$
is a $\pi_1(\Sigma)$-space. Here, $\pi_1(\Sigma)$ acts on
$\widehat{M}$ as the deck transformation and the action of
$\pi_1(\Sigma)$ on $\mathcal{X}$ by conjugation, via the
representations $\varrho$. Thus, under this action
\[\mathcal{X}_{\varrho}:=\widehat{M}\times_{\varrho} \mathcal{X}=(\widehat{M}\times \mathcal{X})/\pi_1(\Sigma).\]

Note that the projection of $\widehat{M} \times \mathcal{X}$ onto
$\widehat{M}$ descends to a map $\mathcal{X}_{\varrho}\rightarrow M$,
which gives $\widehat{M}\times_{\varrho} \mathcal{X}$ the structure
of a flat $\mathcal{X}$-bundle over $M$.

Clearly, the geodesic flow $g_t$ can be lifted to a flow on
$\widehat{M}\times \mathcal{X}$ by defining
\[G_t (m,x):=(g_t m,x).\]

Let us also note that the resulting flow is invariant under the
$\pi_1 (\Sigma)$ action. Therefore, it defines a flow on
$\mathcal{X}_\varrho$, which we denote by the same symbol $G_t$
lifting geodesic flow $g_t$.

We say that a representation $\pi_1(\Sigma)\rightarrow G$ is a
$(P^+,P^-)-$ Anosov, if the following two conditions are
satisfied:

1) there is a section $\sigma :M\rightarrow \mathcal{X}_\varrho$ of
the flat bundle $\mathcal{X}_\varrho$ which is flat along the flow
lines, namely, the restriction of $\sigma$ to any geodesic leaf is
flat,

2) the lifted action of the geodesic flow $g_t$ on
$\sigma^*E^+,\sigma^*E^-$ is respectively expanding, contracting.
More precisely, for some continuous family of norms on the fibers
of the $\sigma^*E^+,\sigma^*E^-$ bundles, the
expanding-contracting properties are fulfilled.\\

Let $\text{Rep}_{\text{Anosov}}(\pi_1(\Sigma),G)$ be the set of all Anosov representations. It was proved in \cite{Labourie2006} by F. Labourie that $\text{Rep}_{\text{Anosov}}(\pi_1(\Sigma),G)$ is open in $\text{Rep}(\pi_1(\Sigma),G).$  He also  proved that every such representation 1-1, discerete, irreducible, and purely loxodromic.

\section{Main Theorems}
Let $\Sigma$ be a closed oriented Riemann surface with  genus $g\geq
2,$ $\widetilde{\Sigma}$ be the universal covering of $\Sigma.$
Let $G\in \{\text{PSp}(2n,\mathbb{R})(n\geq 2), \text{PSO}(n,n+1) (n\geq 2),\text{PSO}(n,n+1)(n\geq 3)\}$ and $\mathcal{G}$ be the corresponding Lie algebra with the non-degenerate Killing form $B.$

For a representation $\varrho:\pi_1(\Sigma)\to
G,$ consider the associated adjoint
bundle $E_{\varrho}=\widetilde{\Sigma}\times
\mathcal{G}/\sim$ over $\Sigma.$ Here, for all
$\gamma \in \pi_1(S),$ $(\gamma\cdot x,\gamma\cdot t)\sim (x,t),$
the action of $\gamma$ in the first component as a deck
transformation and in the second component as conjugation by
$\varrho(\gamma).$

Suppose that $K$ is a cell-decomposition of $\Sigma$ so that the
adjoint bundle $E_{\varrho}$ over $\Sigma$ is trivial over each
cell. Let $\widetilde{K}$ be the lift of $K$ to the universal
covering $\widetilde{\Sigma}$ of $\Sigma.$ Let
$\mathbb{Z}[\pi_1(\Sigma)]=\left\{\sum_{i=1}^p m_i\gamma_i\; ;
m_i\in \mathbb{Z},\; \gamma_i\in \pi_1(\Sigma),\; p\in
\mathbb{N}\right\}$ be the integral group ring. Let
$C_{\ast}(K;\mathcal{G}_{\mathrm{Ad}\circ
\varrho})
 =C_{\ast}(\widetilde{K};\mathbb{Z} )\displaystyle\otimes_{\varrho}\mathcal{G})=\displaystyle C_{\ast}(\widetilde{K};\mathbb{Z})\displaystyle\otimes \mathcal{G}/\sim ,$
 where $\sigma\otimes t$ and all the elements in orbit
 $\{\gamma\cdot\sigma\otimes\gamma\cdot t;\; \gamma\in \pi_1(\Sigma)\}$ are
 identified, and where $\pi_1(\Sigma)$ acts on
$\widetilde{\Sigma}$ by the deck transformation and the action of
$\pi_1(\Sigma)$ on $\mathcal{G}$ is by
conjugation.

We have
$$
0\to C_2(K;\mathcal{G}_{\mathrm{Ad}\circ
\varrho})\stackrel{\partial_2\otimes \mathrm{id}}{\longrightarrow }
C_1(K;\mathcal{G}_{\mathrm{Ad}\circ
\varrho})\stackrel{\partial_1\otimes \mathrm{id}}{\longrightarrow}
C_0(K;\mathcal{G}_{\mathrm{Ad}\circ \varrho})\to
0.
$$
Here, $\partial_p$ is the usual boundary operator. Let
$H_{\ast}(K;\mathcal{G}_{\mathrm{Ad}\circ
\varrho})$ denote the homologies of the above chain complex. The
cochains
$C^{\ast}(K;\mathcal{G}_{\mathrm{Ad}\circ
\varrho})$ yield that
$H^\ast(K;\mathcal{G}_{\mathrm{Ad}\circ
\varrho}).$ Here,
$C^{\ast}(K;\mathcal{G}_{\mathrm{Ad}\circ
\varrho})$ denotes the set of $\mathbb{Z}[\pi_1(\Sigma )]$-module
homomorphisms from $C_{\ast}(\widetilde{K};\mathbb{Z})$ to
$\mathcal{G}.$ For more information, we refer the
reader to  \cite{Porti,SozOJM,Witten}, and the references therein.

We say that $\varrho:\pi_1(\Sigma)\to G$
is \emph{purely loxodromic}, if for every non-trivial $\gamma\in
\pi_1(\Sigma),$ the eigenvalues of $\varrho(\gamma)$ are real with
multiplicity $1.$

Suppose that $\varrho:\pi_1(\Sigma)\to
G$ is purely loxodromic. Consider chain
complex
$$
0\to C_2(K;\mathcal{G}_{\mathrm{Ad}\circ
\varrho})\stackrel{\partial_2\otimes \mathrm{id}}{\longrightarrow }
C_1(K;\mathcal{G}_{\mathrm{Ad}\circ
\varrho})\stackrel{\partial_1\otimes \mathrm{id}}{\longrightarrow}
C_0(K;\mathcal{G}_{\mathrm{Ad}\circ \varrho})\to
0.
$$
Let $e^p_j$ be the $p-$cells of $K$ which gives us a
$\mathbb{Z}-$basis for $C_p(K;\mathbb{Z}).$  Fix a lift
$\widetilde{e}^p_j$ of $e^p_j,$ $j=1,\ldots, m_p.$ Then, $c_p=\{
\widetilde{e}^p_j\}_{j=1}^{m_p}$ is a
$\mathbb{Z}[\pi_1(\Sigma)]-$basis for
$C_p(\widetilde{K};\mathbb{Z}).$ Let
$\mathcal{A}=\{\mathfrak{a}_k\}_{k=1}^{\dim
\mathcal{G}}$ be an $\mathbb{R}-$basis of the
semisimple Lie algebra $\mathcal{G}$ so that the
matrix of the Killing form $B$ is the diagonal matrix
$Diag(\stackrel{p}{1,\ldots,1},\stackrel{r}{-1,\ldots,-1}),$ where
$p+r=\dim\mathcal{G}.$ Such a basis is called
 $B-$\emph{orthonormal basis}. Then, $\mathbf{c}_p=c_p\otimes_{\varrho}
\mathcal{A}$ is an $\mathbb{R}-$basis for
$C_p(K;\mathcal{G}_{\mathrm{Ad}\circ \varrho})$
and called  a \emph{geometric basis} for
$C_p(K;\mathcal{G}_{\mathrm{Ad}\circ \varrho}).$

Let us assume that $\mathbf{h_p}$ is an $\mathbb{R}$-basis for
$H_p(K;\mathcal{G}_{\mathrm{Ad}\circ \varrho})$ then 
$\mathbb{T}(C_{\ast}(K;\mathcal{G}_{\mathrm{Ad}\circ
\varrho}),\{c_p\otimes_{\varrho}
\mathcal{A}\}_{p=0}^2,\{\mathbf{h}_p\}_{p=0}^2)$ is called the
\emph{Reidemeister torsion} of the triple $K,$ $\mathrm{Ad}\circ
{\varrho},$ and $\{\mathbf{h}_p\}_{p=0}^2.$

The independence of the Reidemeister torsion of $\mathcal{A},$ lifts
$\widetilde{e}^p_j,$ conjugacy class of $\varrho,$ and of the
cell-decomposition follows by similar arguments given in
\cite{Milnor}, \cite[Lemma 1.4.2., Lemma 2.0.5.]{SozOJM}.  For the
sake of completeness, the independence of $\mathcal{A},$ lifts
$\widetilde{e}^p_j,$ and conjugacy class of $\varrho$ will be
explained below and for the independence of the cell-decomposition,
the reader is referred to \cite[Lemma 2.0.5.]{SozOJM}.
\begin{prop}\label{Well-definitenessOfRiedemeisterTorsionOfRepresentation}
$\mathbb{T}(C_{\ast}(K;\mathcal{G}_{\mathrm{Ad}\circ
\varrho}),\{c_p\otimes_{\varrho}
\mathcal{A}\}_{p=0}^2,\{\mathbf{h}_p\}_{p=0}^2)$ is independent of
$\mathcal{A},$ lifts $\widetilde{e}^p_j,$ conjugacy class of
$\varrho,$ and the cell-decomposition $K.$
\end{prop}
\begin{proof}
Let $\mathcal{A}'$ be another $B-$orthonormal basis of
$\mathcal{G}.$ From change-base-formula
(\ref{change-base-formula}) of Reidemeister torsion it follows that
\begin{align*}
\frac{\mathbb{T}(C_{\ast}(K;\mathcal{G}_{\mathrm{Ad}\circ
\varrho}),\{\mathbf{c}'_p\}_{p=0}^2,\{\mathbf{h}_p\}_{p=0}^2)}
{\mathbb{T}(C_{\ast}(K;\mathcal{G}_{\mathrm{Ad}\circ
\varrho}),\{\mathbf{c}_p\}_{p=0}^2,\{\mathbf{h}_p\}_{p=0}^2)}
=\det(T)^{-\chi(\Sigma)}.
\end{align*}
Here, $\mathbf{c}'_p=c_p\otimes_{\varrho} \mathcal{A}',$ $T$ is the
change-base-matrix from $\mathcal{A}'$ to $\mathcal{A},$ and $\chi$
is the Euler characteristic.

Note that since $\mathcal{A},$ $\mathcal{A}'$ are $B-$orthonormal
bases of $\mathcal{G},$ $\det T$ is $\pm 1.$ The
independence of the Reidemeister torsion from $B-$orthonormal basis
$\mathcal{A}$ follows from the fact that the Euler-characteristic
$\chi(\Sigma)$  of $\Sigma$ is even.

Next, let us fix $\gamma\in\pi_1(\Sigma).$ Assume
$c'_p=\{\widetilde{e}_1^p\cdot\gamma,\widetilde{e}_2^p,\ldots,\widetilde{e}_{m_p}^p\}$
is also a lift of $\{ e_1^p,\ldots,e_{m_p}^p\},$ where only another
lift of $e_1^p$ is considered and  the others are kept the same.
From the tensor product property it follows that
$\widetilde{e}_1^p\cdot\gamma\otimes t=\widetilde{e}_1^p\otimes
\mathrm{Ad}_{\varrho(\gamma)}(t).$ By change-base-formula
(\ref{change-base-formula}), we have
\begin{align*}
\frac{\mathbb{T}(C_{\ast}(K;\mathcal{G}_{\mathrm{Ad}\circ
\varrho}),\{\mathbf{c}'_p\}_{p=0}^2,\{\mathbf{h}_p\}_{p=0}^2)}{\mathbb{T}(C_{\ast}(K;\mathcal{G}_{\mathrm{Ad}\circ
\varrho}),\{\mathbf{c}_p\}_{p=0}^2, \{\mathbf{h}_p\}_{p=0}^2)}
=\det(A).
\end{align*}
Here, $\mathbf{c}_p=c_p\otimes_{\varrho} \mathcal{A},$
$\mathbf{c}'_p=c'_p\otimes_{\varrho} \mathcal{A},$ and $A$ denotes
the matrix of
$\mathrm{Ad}_{\varrho(\gamma)}:\mathcal{G}\to
\mathcal{G}$ with respect to basis $\mathcal{A}$.

To compute the determinant of the matrix of
$\mathrm{Ad}_{\varrho(\gamma)},$ let us consider the basis
$$
\mathcal{B}_{\mathfrak{sp}_{2n}(\mathbb{R})}=\left\{
  \begin{array}{ll}
    E_{ii}-E_{n+i,n+i}, & 1\leq i\leq n, \\
    E_{ij}-E_{n+j,n+i}, & 1\leq i\neq j\leq n,\\
E_{i,n+i}, & 1\leq i\leq n,\\
E_{n+i,i}, & 1\leq i\leq n,\\
E_{i,n+j}+E_{j,n+i}, & 1\leq i<j\leq n,\\
E_{n+i,j}+E_{n+j,i}, & 1\leq i<j\leq n
  \end{array}
\right.
$$
$$
\mathcal{B}_{\mathfrak{so}_{n,n}(\mathbb{R})}=\left\{
  \begin{array}{ll}
    E_{ij}-E_{n+j,n+i}, & 1\leq i\neq j\leq n, \\
    E_{ii}-E_{n+i,n+i}, & 1\leq i\leq n,\\
E_{i,n+j}-E_{j,n+i}, & 1\leq i<j\leq n,\\
E_{n+i,j}-E_{n+j,i}, & 1\leq i<j\leq n,
  \end{array}
\right.
$$
and 
$$
\mathcal{B}_{\mathfrak{so}_{n,n+1}(\mathbb{R})}=\left\{
  \begin{array}{ll}
    E_{ii}-E_{n+i,n+i}, & 2\leq i\leq n+1, \\
    E_{1,n+i+1}-E_{i+1,1}, & 1\leq i\leq n,\\
E_{1,i+1}-E_{n+i+1,1}, & 1\leq i \leq n,\\
E_{i+1,j+1}-E_{n+j+1,n+i+1}, & 1\leq i\neq j\leq n,\\
E_{i+1,n+j+1}-E_{j+1,n+i+1}, & 1\leq i\neq j\leq n,\\
E_{i+n+1,j+1}-E_{j+n+1,i+1}, & 1\leq j\neq i\leq n
  \end{array}
\right.
$$
of $\mathfrak{sp}_{2n}(\mathbb{R}),$ $\mathfrak{so}_{n,n}(\mathbb{R}),$ and $\mathfrak{so}_{n,n+1}(\mathbb{R}),$ respectively. Here,  $E_{ij}$ denotes the
matrix with $1$ in the $ij$ entry and $0$ elsewhere. 

By the
assumption that $\varrho$ is purely loxodromic, we have for each $\gamma\in
\pi_1(\Sigma)$ there is $Q=Q(\gamma)\in
G$ such that
$Q\varrho(\gamma)Q^{-1}=D=\text{Diag}(\lambda_1,\ldots,\lambda_{m}).$
Here, $m$ is equal to $2n$ for $G\in\{\text{PSp}(2n,\mathbb{R}),\text{PSO}(n,n)\},$ 
and for $G=\text{PSO}(n,n+1),$ $m=2n+1.$

Note that by the spectral properties of such diagonalizable  matrices we have
$D=\text{Diag}(\lambda_1,\ldots,\lambda_{n},1/\lambda_1,\ldots,1/\lambda_n)$ for $G\in\{\text{PSp}(2n,\mathbb{R}),\text{PSO}(n,n),\}$ and 
$D=\text{Diag}(1,\lambda_1,\ldots,\lambda_{n},1/\lambda_1,\ldots,1/\lambda_n)$ for $G=\text{PSO}(n,n+1).$
Note also that 
$$DE_{ij}D^{-1}=\frac{\lambda_i}{\lambda_j}E_{ij}.
$$

From this it follows that the matrix of $\mathrm{Ad}_D$ in the basis
$\mathcal{B}_{\mathfrak{sp}_{2n}(\mathbb{R})}$ is the diagonal matrix with the diagonal entries
$$\left\{
  \begin{array}{ll}
1,&1\leq i\leq n\\
 \frac{\lambda_i}{\lambda_j},& 1\leq i\neq
j\leq n\\
  \lambda^2_i,&1\leq i\leq n\\
   \frac{1}{\lambda^2_i},& 1\leq
i\leq n\\
  \lambda_i\lambda_j,&1\leq i<j\leq
n\\
\frac{1}{\lambda^2_i},&1\leq i\leq n\\ 
\frac{1}{\lambda_i\lambda_j},&1\leq i<j\leq n,
\end{array}
\right.
$$
in the basis $\mathcal{B}_{\mathfrak{so}_{n,n}(\mathbb{R})}$ is the diagonal matrix with the diagonal entries
$$\left\{
  \begin{array}{ll}
  \frac{\lambda_i}{\lambda_j},&1\leq i\neq
j\leq n\\
 1,&1\leq i\leq n\\
  \lambda_i\lambda_j,&1\leq i<
j\leq n\\
 \frac{1}{\lambda_i\lambda_j},&1\leq i<
j\leq n,
\end{array}
\right.
$$
in the basis $\mathcal{B}_{\mathfrak{so}_{n,n+1}(\mathbb{R})}$ is the diagonal matrix with the diagonal entries
$$\left\{
  \begin{array}{ll}
  1,&2\leq i \leq n+1\\
   \lambda_{i+1},& 1\leq i\leq n\\
   1/\lambda_{i+1},& 1\leq i\leq n\\
   \frac{\lambda_{i+1}}{\lambda_{j+1}},&1\leq i\neq j\leq n\\
   \lambda_{i+1}\lambda_{j+1},&1\leq i<j\leq n\\\frac{1}{\lambda_{i+1}\lambda_{j+1}},& 1\leq j<i\leq n.
\end{array}
\right.
$$

Note that the
determinant of these diagonal matrices is $1$. Hence, we proved the
independence of the Reidemeister torsion from the lifts.

By the fact that the twisted chains and cochains for conjugate
representations are isomorphic, we also have the independence of
Reidemeister torsion from conjugacy class of $\varrho.$

This is the end of proof
Proposition~\ref{Well-definitenessOfRiedemeisterTorsionOfRepresentation}.
\end{proof}

Let us continue with the following well known result which will be used in the proof of our main theorem (Theorem \ref{MainThm}). For the sake of completeness, we will also give the proof of this auxilary result in details.
\begin{lem}\label{AuxilaryResult}
Let $f:V\times V\to \mathbb{R}$ be a non-degenarate, anti-symmetric  bilinear map on the 
real vector space $V$ of dimension $2n.$ Let $\{ v_1,\ldots,v_{2n} \}$ be a basis of $V$ and let $\{ v^1,\ldots,v^{2n} \}$
be the corresponding dual basis, namely $v^i(v_j)=\delta_{ij}.$ Let $f^{\ast }:V^{\ast }\times V^{\ast } \to \mathbb{R}$ be the dual bilinear map of $f,$ which is defined by $f^{\ast }(v^i,v^j):=f(v_i,v_j).$ If $G(f;\{ v_1,\ldots,v_{2n}\})$ denotes the Gram matrix of $f$ in the basis $\{ v_1,\ldots,v_{2n}\},$ then  $G(f^{\ast };\{ v^1,\ldots,v^{2n}\})G(f;\{ v_1,\ldots,v_{2n}\})^{\mathrm{T}}=\mathrm{I}_{2n\times 2n}.$ Here, $\mathrm{I}_{2n\times 2n }$ is the $n\times n$ identity matrix and ``$\mathrm{T} $'' denotes the transpose of a matrix.
\end{lem}
\begin{proof}
Let us first note that there is a symplectic basis $\{ e_1,\ldots,e_{2n} \}$ of $V$ so that the Gram matrix $G(f;\{ e_1,\ldots,e_{n},e_{n+1},\ldots,e_{2n} \})$ and $G(f^{\ast };\{ e^1,\ldots,e^{n},e^{n+1},\ldots,e^{2n} \})$ are both equal to
$\left(
\begin{array}{cc}
  0_{n\times n} & \mathrm{I}_{n\times n} \\
  -\mathrm{I}_{n\times n} & 0_{n\times n} \\
\end{array}
\right).$ 
Thus, we have

$$G(f^{\ast };\{ e^1,\ldots,e^{n},e^{n+1},\ldots,e^{2n}\})G(f;\{ e_1,\ldots,e_{n},e_{n+1},\ldots,e_{2n}\})^{\mathrm{T}}=I_{2n\times 2n}.$$
  
If $L$ is the change-base matrix from  basis $\{ e_1,\ldots,e_{2n}\}$ to $\{ v_1,\ldots,v_{2n}\}$ of $V$ and if $M$ is the change-base matrix from  basis $\{ e^1,\ldots,e^{2n}\}$ to $\{ v^1,\ldots,v^{2n}\}$ of $V^{\ast },$ then clearly we have $LM^{\mathrm{T}}=\mathrm{I}_{2n\times 2n }.$ Note also that 
$$G(f;\{ v_1,\ldots,v_{n},v_{n+1},\ldots,v_{2n} \})=L^{\mathrm{T}}G(f;\{ e_1,\ldots,e_{n},e_{n+1},\ldots,e_{2n} \})L,$$
$$G(f^{\ast };\{ v^1,\ldots,v^{n},v^{n+1},\ldots,v^{2n} \})=M^{\mathrm{T}}G(f^{\ast };\{ e^1,\ldots,e^{n},e^{n+1},\ldots,e^{2n} \})M.$$
From these it follows that
 $G(f^{\ast };\{ v^1,\ldots,v^{2n}\})G(f;\{ v_1,\ldots,v_{2n}\})^{\mathrm{T}}=\mathrm{I}_{2n\times 2n}.$
  
\noindent   This finishes the proof of Lemma \ref{AuxilaryResult}.
\end{proof}

\begin{thm}\label{MainThm}
Assume that $\Sigma$ is a closed orientable surface of genus $g\geq
2$  and  $\varrho:\pi_1(\Sigma)\to G$
is an irreducible, purely loxodromic representation. Assume also $K$
is a cell-decomposition of $\Sigma,$  $\mathbf{c}_p$ is the
geometric bases of
$C_p(K;\mathcal{G}_{\mathrm{Ad}\circ \varrho}),$
$p=0,1,2,$ and $\mathbf{h}_1$ is a basis for
$H_1(\Sigma;\mathcal{G}_{\mathrm{Ad}\circ
\varrho}).$ Then, the following formula is valid:
$$
\mathbb{T}(C_{\ast}(K;\mathcal{G}_{\mathrm{Ad}\circ
\varrho}),\{\mathbf{c}_p\}_{p=0}^2,\{0,\mathbf{h}_1,0\})=
\sqrt{\det\Omega_{\omega_B}}, $$ where $ \omega_B :
H^1(\Sigma;\mathcal{G}_{\mathrm{Ad}\circ
\varrho})\times
H^1(\Sigma;\mathcal{G}_{\mathrm{Ad}\circ
\varrho})\stackrel{\smile_B}{\longrightarrow}
H^2(\Sigma;\mathbb{R})\stackrel{\int_{\Sigma}}{\longrightarrow}\mathbb{R}
$ is the Atiyah-Bott-Goldman symplectic form for the Lie group $G,$  $\Omega_{\omega_B}$ is
the matrix of  $\omega_B$
in the basis $\mathbf{h}^1,$
 and where $\mathbf{h}^1$ is the
Poincar\'{e} dual basis of
$H^1(\Sigma;\mathcal{G}_{\mathrm{Ad}\circ
\varrho})$ corresponding to $\mathbf{h}_1$ of
$H_1(\Sigma;\mathcal{G}_{\mathrm{Ad}\circ
\varrho}).$
\end{thm}
\begin{proof}
By the invariance of the Cartan-Killing form $B$ of
$\mathcal{G}$ under conjugation, for $k=0,1,2,$
we have the non-degenerate form, the \emph{Kronecker pairing,}
$$<\cdot,\cdot>:C^k(K;\mathcal{G}_{\mathrm{Ad}\circ
\varrho})\times
C_k(K;\mathcal{G}_{\mathrm{Ad}\circ \varrho})\to
\mathbb{R}$$ defined by $B(t,\theta(\sigma)),$ $\theta\in
C^k(K;\mathcal{G}_{\mathrm{Ad}\circ \varrho}),$
$\sigma\otimes_\varrho t \in
C_k(K;\mathcal{G}_{\mathrm{Ad}\circ \varrho}).$
It is extended to
$$<\cdot,\cdot>:H^k(\Sigma;\mathcal{G}_{\mathrm{Ad}\circ
\varrho})\times
H_k(\Sigma;\mathcal{G}_{\mathrm{Ad}\circ
\varrho})\to \mathbb{R}.$$

Invariance of $B$ under the conjugation and the non-degeneracy of
$B$ yield the cup product
$$\smile_B:C^k(K;\mathcal{G}_{\mathrm{Ad}\circ
\varrho})\times C^{\ell
}(K;\mathcal{G}_{\mathrm{Ad}\circ \varrho})\to
C^{k+{\ell }}(K;\mathbb{R})$$ defined by
$(\theta_k\smile_B\theta_{\ell })(\sigma_{k+{\ell }})=
B(\theta_k((\sigma_{k+{\ell }})_\mathrm{front})),\theta_{\ell
}((\sigma_{k+{\ell }})_\mathrm{back}).$ Clearly, $\smile_B$ has the
extension $$
  \smile_B   :  H^k(\Sigma;\mathcal{G}_{\mathrm{Ad}\circ \varrho})   \times   H^{\ell}(\Sigma;\mathcal{G}
  _{\mathrm{Ad}\circ \varrho})   \to   H^{k+\ell}(\Sigma;\mathbb{R}).
$$

Let us denote by $K'$ the dual cell-decomposition of $\Sigma$
corresponding to the cell decomposition $K.$ Suppose that cells
$\sigma\in K,$ $\sigma'\in K'$ meet at most once and also the
diameter of each cell is less than, say, half of the injectivity
radius of $\Sigma.$ By the fact that the Reidemeister torsion is
invariant under subdivision, this assumption is not loss of
generality. Let $c'_p$ be the basis of
$C_p(\widetilde{K'};\mathbb{Z})$ corresponding to the basis $c_p$ of
$C_p(\widetilde{K};\mathbb{Z}),$ and let
$\mathbf{c}'_p=c'_p\otimes_{\varrho} \mathcal{A}$ be the
corresponding basis for
$C_p(K';\mathcal{G}_{\mathrm{Ad}\circ \varrho}).$

We have the intersection form
\begin{equation}\label{intersection form}
(\cdot,\cdot)_{k,2-k}:C_k(K;\mathcal{G}_{\mathrm{Ad}\circ
\varrho})\times
C_{2-k}(K';\mathcal{G}_{\mathrm{Ad}\circ
\varrho})\to\mathbb{R}
\end{equation}
 defined by $$(\sigma_1\otimes
t_1,\sigma_2\otimes t_2)_{k,2-k}=\sum_{\gamma\in\pi_1(\Sigma)}
\sigma_1.(\gamma\cdot\sigma_2)\; B(t_1,\gamma\cdot t_2),$$ where
$``."$ is the intersection number pairing. Note that
$(\cdot,\cdot)_{k,2-k}$ are $\partial-$compatible because the
intersection number pairing $``."$ is compatible with the usual
boundary operator in the sense $(\partial\alpha).\beta =
(-1)^{|\alpha|}\alpha .(\partial\beta),$ where $|\alpha|$ is the
dimension of the cell $\alpha.$ Since the intersection number form
$``."$ is anti-symmetric  and $B$ is invariant under adjoint action,
then $(\cdot,\cdot)_{k,2-k}$ is anti-symmetric.

By the independence of the twisted homologies from the
cell-decomposition, we get the non-degenerate anti-symmetric form
\begin{equation}\label{intersection formH}
(\cdot,\cdot)_{k,2-k}:H_k(\Sigma;\mathcal{G}_{\mathrm{Ad}\circ
\varrho})\times
H_{2-k}(\Sigma;\mathcal{G}_{\mathrm{Ad}\circ
\varrho})\to \mathbb{R}.
\end{equation}
Combining the isomorphisms induced by the Kronecker pairing and the
intersection form, we obtain the Poincar\'{e} duality isomorphisms
$$
\mathrm{PD}:H_k(\Sigma;\mathcal{G}_{\mathrm{Ad}\circ
\varrho})\stackrel{}{\cong}
H_{2-k}(\Sigma;\mathcal{G}_{\mathrm{Ad}\circ
\varrho})^{\ast}\stackrel{}{\cong}
H^{2-k}(\Sigma;\mathcal{G}_{\mathrm{Ad}\circ
\varrho}).$$ Thus, for $k=0,1,2,$ we have the following commutative
diagram
\begin{equation*}
\begin{array}{ccccc}
  H^{2-k}(\Sigma;\mathcal{G}_{\mathrm{Ad}\circ \varrho}) & \times & H^k(\Sigma;\mathcal{G}_{\mathrm{Ad}\circ \varrho}) & \stackrel{\smile_B}{\longrightarrow} & H^2(\Sigma;\mathbb{R}) \\
  \big\uparrow \small\mathrm{PD}&   & \big\uparrow\small\mathrm{PD} & \circlearrowleft & \big\uparrow \\
  H_k(\Sigma;\mathcal{G}_{\mathrm{Ad}\circ \varrho}) & \times & H_{2-k}(\Sigma;\mathcal{G}_{\mathrm{Ad}\circ \varrho}) & \stackrel{(,)_{k,2-k}}{\longrightarrow} & \mathbb{R}.
\end{array}
\end{equation*}
Here, the isomorphism $\mathbb{R}\to H^2(\Sigma;\mathbb{R})$ sends
$1$ to the fundamental class of $H^2(\Sigma;\mathbb{R}).$

By the irreducibility of $\varrho,$ we have
$H_0(\Sigma;\mathcal{G}_{\mathrm{Ad}\circ
\varrho}),$
$H_2(\Sigma;\mathcal{G}_{\mathrm{Ad}\circ
\varrho}),$
$H^0(\Sigma;\mathcal{G}_{\mathrm{Ad}\circ
\varrho}),$ and
$H^2(\Sigma;\mathcal{G}_{\mathrm{Ad}\circ
\varrho})$ are all zero. Hence,
\begin{equation}\label{diagram}
\begin{array}{ccccc}
  H^{1}(\Sigma;\mathcal{G}_{\mathrm{Ad}\circ \varrho}) & \times & H^1(\Sigma;\mathcal{G}_{\mathrm{Ad}\circ \varrho}) & \stackrel{\smile_B}{\longrightarrow} & H^2(\Sigma;\mathbb{R}) \\
  \big\uparrow \small\mathrm{PD}&   & \big\uparrow\small\mathrm{PD} & \circlearrowleft & \big\uparrow \\
  H_1(\Sigma;\mathcal{G}_{\mathrm{Ad}\circ \varrho}) & \times & H_{1}(\Sigma;\mathcal{G}_{\mathrm{Ad}\circ \varrho}) & \stackrel{(,)_{1,1}}{\longrightarrow} & \mathbb{R}.
\end{array}
\end{equation}

Recall that $ \omega_B :
H^1(\Sigma;\mathcal{G}_{\mathrm{Ad}\circ
\varrho})\times
H^1(\Sigma;\mathcal{G}_{\mathrm{Ad}\circ
\varrho})\stackrel{\smile_B}{\longrightarrow}
H^2(\Sigma;\mathbb{R})\stackrel{\int_{\Sigma}}{\longrightarrow}\mathbb{R}
$
is called the Atiyah-Bott-Goldman symplectic form for the Lie group $G.$ Note also that from equation (\ref{diagram}), $ \omega_B $ is nothing but the dual of the intersection pairing $(,)_{1,1}.$\\

Let $C_{p}=C_p(K;\mathcal{G}_{\mathrm{Ad}\circ
\varrho}),$
$C'_{p}=C_p(K';\mathcal{G}_{\mathrm{Ad}\circ
\varrho}),$ and $D_{p}=C_{\ast}\oplus C'_{\ast}.$ Consider the
intersection form (\ref{intersection form}), define it on $C_p\times
C_{2-p}$ and $C'_p\times C'_{2-p}$ as $0.$ Let
$\omega_{p,2-p}:D_p\times D_{2-p}\to \mathbb{R}$  be defined by
using $(\cdot,\cdot)_{p,2-p}.$ Then, $D_{\ast}$ becomes a symplectic
chain complex. Note that $\omega_{1,1}:H_1(D_{\ast})\times
H_1(D_{\ast})\to \mathbb{R}$ is equal to
$\left(%
\begin{array}{cc}
  0&(\cdot,\cdot)_{1,1} \\
  -(\cdot,\cdot)_{1,1}&0
\end{array}%
\right)$ and $(\cdot,\cdot)_{1,1}$ is  the intersection form
(\ref{intersection formH}) for $k=1.$ Then, from Lemma
\ref{sumlemma}, Theorem~\ref{Tor_Of_General_Symplectic},
independence of the Reidemeister torsion from the cell-decomposition
of $\Sigma,$ and the fact that $D_{\ast}$ is a symplectic chain
complex it follows that
\begin{equation}\label{nby0}
 \mathbb{T}(D_{\ast},\{ \mathbf{c}_p\oplus
\mathbf{c}_p'\}_{p=0}^2,\{0\oplus 0,\mathbf{h}_1\oplus
\mathbf{h}_1,0\oplus 0\}) =\sqrt{\Delta_{1,1}(D_{\ast })}^{(-1)}.
\end{equation}

Since the intersection form (\ref{intersection formH}) for $k=1$ is
non-degenerate, then equation (\ref{nby0}) becomes
\begin{equation}\label{nby}
\mathbb{T}(D_{\ast},\{\mathbf{c}_p\oplus
\mathbf{c}'_p\}_{p=0}^2,\{0\oplus 0, \mathbf{h}_1\oplus
\mathbf{h}_1,0\oplus 0\})=\Delta_{1,1}(C_{\ast })^{(-1)}.
\end{equation}

Let us consider the short-exact sequence
$$0\to C_{\ast}\hookrightarrow D_{\ast}=C_{\ast}\oplus
C'_{\ast}\twoheadrightarrow C'_{\ast}\to 0.$$ Here,
$C_{\ast}\hookrightarrow D_{\ast}$ denotes the inclusion,
$D_{\ast}\twoheadrightarrow C'_{\ast}$  denotes the projection.
Clearly, the bases $\mathbf{c}_p$ of $C_p,$ $\mathbf{c}_p \oplus
\mathbf{c}_p'$ of $D_{\ast},$ and $\mathbf{c}_p'$ of $C'_{\ast}$ are
compatible. By Lemma \ref{sumlemma} and the independence of the
Reidemeister torsion from the cell-decomposition of $\Sigma,$ we
have \begin{equation}\label{z1} \mathbb{T}(D_{\ast},\{
\mathbf{c}_p\oplus \mathbf{c}_p'\}_{p=0}^2,\{0\oplus
0,\mathbf{h}_1\oplus \mathbf{h}_1,0\oplus
0\})=\left(\mathbb{T}(C_{\ast},\{\mathbf{c}_p\}_{p=0}^2,\{0,\mathbf{h}_1,0\})\;\right)^2.
\end{equation}

Thus, combining equations   (\ref{nby}) and (\ref{z1}), we obtain
\begin{equation}\label{nby1}\mathbb{T}(C_{\ast},\{\mathbf{c}_p\}_{p=0}^2,\{0,\mathbf{h}_1,0\})=\sqrt{\Delta_{1,1}(C_{\ast}) }^{(-1)}.
\end{equation}

The fact that $\smile_B$ is the dual of the intersection pairing $(,)_{1,1}$  and Lemma \ref{AuxilaryResult} yield that 
\begin{equation}\label{nby1}\mathbb{T}(C_{\ast},\{\mathbf{c}_p\}_{p=0}^2,\{0,\mathbf{h}_1,0\})=\sqrt{\det \Omega_{\omega_B} }.
\end{equation}

This concludes the proof of
Theorem \ref{MainThm}.
\end{proof}

\begin{cor}\label{cor}Since every Anosov representation is 1-1, discerete, irreducible, and purely loxodromic \cite{Labourie2006}, then Theorem \ref{MainThm} also holds for Anosov representations.
\end{cor}

\section{Application:A Volume element on some Hitchin components}
For a closed oriented Riemann surface $\Sigma$ with genus $g>1$ and
a semi-simple Lie group $G,$ let us denote by
$\mathrm{Hom}(\pi_1(\Sigma),G)$  the set of all homomorphisms from
the fundamental group $\pi_1(\Sigma)$ of $\Sigma$ to $G.$

Let us consider the orbit space $\mathrm{Hom}(\pi_1(\Sigma),G)/G,$
where the action of $G$ on $\mathrm{Hom}(\pi_1(\Sigma),G)$ by
conjugation i.e. $g\cdot\varrho(\gamma)=g\varrho(\gamma)g^{-1}$, for 
$g\in G,$ $\varrho\in \mathrm{Hom}(\pi_1(\Sigma),G),$ and $\gamma\in
\pi_1(\Sigma).$  It is well known that this is a real analytic
variety. Moreover, for algebraic $G,$
$\mathrm{Hom}(\pi_1(\Sigma),G)/G$ is also  algebraic. This orbit space is
not necessarily Hausdorff (cf., e.g. \cite{Goldman}) but the space
$\mathrm{Rep}(\pi_1(\Sigma),G)=\mathrm{Hom}^{+}(\pi_1(\Sigma),G)/G$
of all reductive representations of $\pi_1(\Sigma)$ in $G$ is
Hausdorff. A reductive representation is the one that once composed
with adjoint representation of $G$ on its Lie algebra $\mathcal{G}$
is a sum of irreducible representations.

\emph{Teichm\"{u}ller space} $\mathrm{Teich}(\Sigma)$ of $\Sigma$
is the space of isotopy classes of complex structures on
$\Sigma.$ A \emph{complex structure} on $\Sigma$ is a homotopy
equivalence of a homeomorphism $f:\Sigma\to S.$ Here, $S$ is a
Riemann surface, and two such homeomorphisms $f:\Sigma\to S,$
$f':\Sigma\to S'$ are said to be equivalent, if there exists a
conformal diffeomorphism $g:S\to S' $ so that $(f')^{-1}\circ
g\circ f$ is isotopic to the identity map on $\Sigma.$

One can lift a complex structure on $\Sigma$ to a complex structure
on the universal covering $\widetilde{\Sigma}$ of $\Sigma.$ By the
Uniformization Theorem, $\widetilde{\Sigma}$ is biholomorphic to the
upper half-plane $\mathbb{H}^2\subset \mathbb{C}.$ It is well known
that each biholomorphic homeomorphism of $\mathbb{H}^2$ is of the
form $f(z)=(az+b)/(cz+d)$ with $a,b,c,d\in \mathbb{R},$ $ad-bc=1.$
This yields a discrete, faithful homomorphism from $\pi_1(\Sigma)$
to $\mathrm{PSL}(2,\mathbb{R}).$ This homomorphism is also well
defined up to conjugation by the orientation preserving isometries
of $\mathbb{H}^2.$ Thus, one can identify $\mathrm{Teich}(\Sigma)$
with the \emph{Fricke space}, i.e. the set
 $\mathrm{Rep}_{\mathrm{df}}(\pi_1(\Sigma),\mathrm{PSL}(2,\mathbb{R}))$ of
discrete faithful representations from $\pi_1(\Sigma)$ to
$\mathrm{PSL}(2,\mathbb{R}).$

Fricke space is a connected component of
$\mathrm{Rep}(\pi_1(\Sigma),\mathrm{PSL}(2,\mathbb{R})).$ Openness
follows from \cite{Weil1}, closedness from
\cite{Chuckrow,Ragunathan}, and connectedness from the
Uniformization Theorem together with the identification of
$\mathrm{Teich}(\Sigma)$ as a cell.

For a finite cover $G$ of $\mathrm{PSL}(2,\mathbb{R}),$ W. Goldman
investigated the connected components of the representation space
$\mathrm{Hom}(\pi_1(\Sigma),G)/G$ \cite{GoldmanTopComp}. He proved
that there exist $4g-3$ connected components of
$\mathrm{Hom}(\pi_1(\Sigma),\mathrm{PSL}(2,\mathbb{R}))/\mathrm{PSL}(2,\mathbb{R}).$
There exist two homeomorphic components, called Teichm\"{u}ller
spaces, which are homeomorphic to
$\mathbb{R}^{|\chi(\Sigma)|\dim\mathrm{PSL}(2,\mathbb{R})}.$

For a split real form $G$ of a semi-simple Lie group, N. Hitchin
investigated the connected components of
$\mathrm{Rep}(\pi_1(\Sigma),G)$ in \cite{Hitchin} by using
techniques of Higgs bundle.  He proved that there exists an
interesting connected component not detected by characteristic
classes. He called it as \emph{Teichm\"{u}ller component} but it is
called now \emph{Hitchin component}.

A \emph{Hitchin component}
$\mathrm{Rep}_{\mathrm{Hitchin}}(\pi_1(\Sigma),G)$ of
$\mathrm{Rep}(\pi_1(\Sigma),G)$ is the connected component
containing Fuchsian representations, i.e. representations of the
form $\varrho\circ \imath,$ where $\varrho:\pi_1(\Sigma)\to
\mathrm{PSL}(2,\mathbb{R})$ is Fuchsian,
$\imath:\mathrm{PSL}(2,\mathbb{R})\to G$ is the representation
corresponding to the $3-$dimensional principal subgroup of B.
Kostant \cite{Kostant}. For $G=\mathrm{PSp}(2n,\mathbb{R}),$
$\imath$ denotes the $2n-$dimensional irreducible representation
corresponding to symmetric power
$\mathrm{Sym}^{2n-1}(\mathbb{R}^2).$

This enables one to identify the Fricke space and thus
$\mathrm{Teich}(\Sigma)$ by a subset of
$\mathrm{Rep}(\pi_1(\Sigma),G).$ N. Hitchin proved in \cite{Hitchin} that each Hitchin component is homoeomorphic to a ball of dimension
$(6g-6)\dim G.$ Recall  that it was proved by F. Labourie in \cite{Labourie2006} that the set $\text{Rep}_{\text{Hitchin}}(\pi_1(\Sigma),G)$  of Hitchin representations is a subset of $\text{Rep}_{\text{Anosov}}(\pi_1(\Sigma),G).$ \\

Applying Theorem \ref{MainThm} and Corollary \ref{cor}, we have the following result.
\begin{cor}\label{MainCorollary}
Let $\Sigma$ be a closed orientable surface of genus $g\geq
2$  and  $\varrho$
be in $\mathrm{Rep}_{\mathrm{Hitchin}}(\pi_1(\Sigma),G).$ Let  $K$
be a cell-decomposition of $\Sigma,$  $\mathbf{c}_p$ be the
geometric bases of
$C_p(K;\mathcal{G}_{\mathrm{Ad}\circ \varrho}),$
$p=0,1,2,$ and $\mathbf{h}_1$ is a basis for
$H_1(\Sigma;\mathcal{G}_{\mathrm{Ad}\circ
\varrho}).$ Then, we have
$$
\mathbb{T}(C_{\ast}(K;\mathcal{G}_{\mathrm{Ad}\circ
\varrho}),\{\mathbf{c}_p\}_{p=0}^2,\{0,\mathbf{h}_1,0\})=
\sqrt{\det\Omega_{\omega_B}}.$$ Moreover, it is  
a volume element 
on  the Hitchin component
$\mathrm{Rep}_{\mathrm{Hitchin}}(\pi_1(\Sigma),G).$
\noindent Here, $G$ is one of $\{\text{PSp}(2n,\mathbb{R})(n\geq 2), \text{PSO}(n,n+1) n\geq 2),\text{PSO}(n,n+1)(n\geq 3\}$ and $\mathcal{G}$ is the corresponding Lie algebra with the non-degenerate Killing form $B$ and  $\omega_B :
H^1(\Sigma;\mathcal{G}_{\mathrm{Ad}\circ
\varrho})\times
H^1(\Sigma;\mathcal{G}_{\mathrm{Ad}\circ
\varrho})\stackrel{\smile_B}{\longrightarrow}
H^2(\Sigma;\mathbb{R})\stackrel{\int_{\Sigma}}{\longrightarrow}\mathbb{R}
$ is the Atiyah-Bott-Goldman symplectic form for $G.$
\end{cor}
\begin{proof}
From the fact that $\varrho$ belongs to
 $\mathrm{Rep}_{\mathrm{Hitchin}}(\pi_1(\Sigma),G)$ it follows that it is dicrete, faithfull, irreducible, and purely loxodromic (cf.
\cite{BurgerIozziWienhard,V.FockA.Goncharov,Labourie2006}). The irreducibility yields that
$H^0(\Sigma,\mathcal{G}_{\mathrm{Ad}\circ
\varrho})$ and $H^2(\Sigma,\mathcal{G}_{\mathrm{Ad}\circ
\varrho})$ are both  zero.   Hence, by Theorem \ref{MainThm}, we obtain
$$\mathbb{T}(C_{\ast}(K;\mathcal{G}_{\mathrm{Ad}\circ
\varrho}),\{\mathbf{c}_p\}_{p=0}^2,\{0,\mathbf{h}_1,0\})=
\sqrt{\det\Omega_{\omega_B}}.$$
It is well known that
$H^1(\Sigma,\mathcal{G}_{\mathrm{Ad}\circ
\varrho}),$
$H_1(\Sigma,\mathcal{G}_{\mathrm{Ad}\circ
\varrho})$ can be identified respectively with the tangent space
$T_{\varrho }
\mathrm{Rep}_{\mathrm{Hitchin}}(\pi_1(\Sigma),G),$
cotangent space  $T^{\ast}_{\varrho }
\mathrm{Rep}_{\mathrm{Hitchin}}(\pi_1(\Sigma),G)$
of
$\mathrm{Rep}_{\mathrm{Hitchin}}(\pi_1(\Sigma),G)$
(cf., e.g. \cite{Goldman}). Recall also that Reidemeister torsion
$\mathbb{T}(A_{\ast})$ of a general chain complex $A_{\ast}$ of length $n$
belongs to $\otimes_{p=0}^n
(\det(H_p(A_{\ast})))^{(-1)^{p+1}}$ (\cite{SozOJM,Witten}). Here,
$\det(H_p(A_{\ast}))$ is the top exterior power
$\bigwedge^{\dim_{\mathbb{R}} H_p(A_{\ast})
}H_p(A_{\ast})$   of $H_p(A_{\ast})$
and $\det(H_p(A_{\ast}))^{-1}$ is the dual of
$\det(H_p(A_{\ast})).$  Thus, we get a volume element on  the Hitchin component
$\mathrm{Rep}_{\mathrm{Hitchin}}(\pi_1(\Sigma),G)$
of $\mathrm{Rep}(\pi_1(\Sigma),G).$
\end{proof}

\noindent Let us note that since $\mathrm{Teich}(\Sigma)\subset
\mathrm{Rep}_{\mathrm{Hitchin}}(\pi_1(\Sigma),G),$  
Corollary \ref{MainCorollary} is also valid for $\mathrm{Teich}(\Sigma)$ representations.

For the isomorphism
$\mathrm{T}_{\varrho}\mathrm{Teich}(\Sigma)\cong
H^1(\Sigma;\mathcal{G}_{\mathrm{Ad}\circ \varrho}),$ in \cite{Goldman}, Goldman proved that
$\omega_{\mathrm{PSL(2,\mathbb{R})}}:H^{1}(\Sigma;\mathcal{G}_{\mathrm{Ad}\circ
\varrho}) \times H^1(\Sigma;\mathcal{G}_{\mathrm{Ad}\circ \varrho}) \to
\mathbb{R}$   and Weil-Petersson $2-$form differ only by a
constant multiple. More precisely,
$$\omega_{\mathrm{WP}}=-8\omega_{\mathrm{PSL(2,\mathbb{R})}}.$$

 Bonahon parametrized the Teichm\"{u}ller space of $\Sigma$ by using a
maximal geodesic lamination $\lambda$ on $\Sigma$  \cite{Bonahon}. Geodesic
laminations are generalizations of deformation classes of simple
closed curves on $\Sigma.$ More precisely, a geodesic lamination
$\lambda$ on the surface $\Sigma$ is by definition a closed subset
of $\Sigma$ which can be decomposed into family of disjoint simple
geodesics, possibly infinite, called its \emph{leaves.} The
geodesic lamination is \emph{maximal} if it is maximal with
respect to inclusion; this is equivalent to the property that the
complement $\Sigma-\lambda$ is union of finitely many triangles
with vertices at infinity.

\begin{figure}[h]
\begin{center}
\includegraphics[scale=0.5]{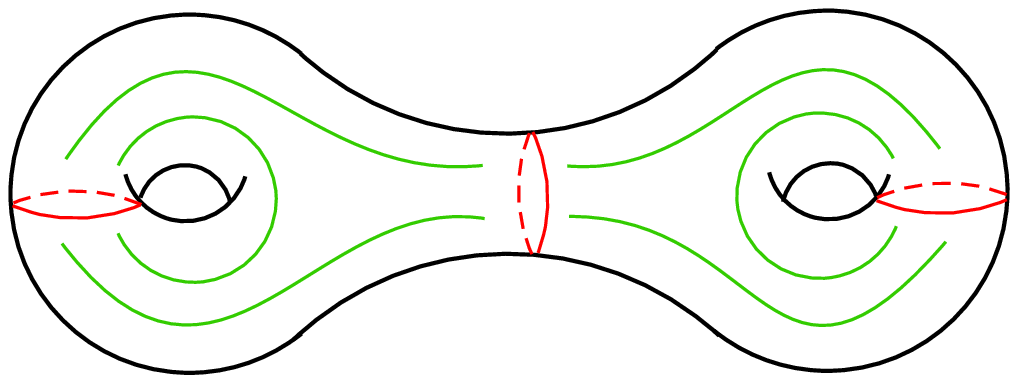}
\end{center}
\end{figure}

 The real-analytical parametrization given by Bonahon identifies
$\mathrm{Teich}(\Sigma)$ to an open convex cone in the vector
space ${\mathcal H}(\lambda,\mathbb{R})$ of all \emph{transverse
cocycles} for $\lambda.$ In particular, at each $\varrho\in
\mathrm{Teich}(\Sigma),$ the tangent space
$\text{T}_{\varrho}\mathrm{Teich}(\Sigma)$ is now identified with
${\mathcal H}(\lambda,\mathbb{R}),$ which is a real vector space
of dimension $3|\chi(\Sigma)|.$

A \emph{transverse cocycle} $\sigma$ for $\lambda$ on $\Sigma$ is
a real-valued function on the set of all arcs $k$ transverse to
(the leaves) of $\lambda$ with the following properties:
\begin{itemize}
\item $\sigma$ is finitely additive, i.e.
$\sigma(k)=\sigma(k_1)+\sigma(k_2),$ whenever the arc $k$
transverse to $\lambda$ is decomposed into two subarcs $k_1,k_2$
with disjoint interiors,
\item $\sigma$ is invariant under the homotopy of arcs transverse to
$\lambda,$ i.e. $\sigma(k)=\sigma(k')$ whenever the transverse arc
$k$ is deformed to arc $k'$ by a family of arcs which are all
transverse to the leaves of  $\lambda.$
\end{itemize}

\begin{figure}[h]
\begin{center}
\includegraphics[scale=0.5]{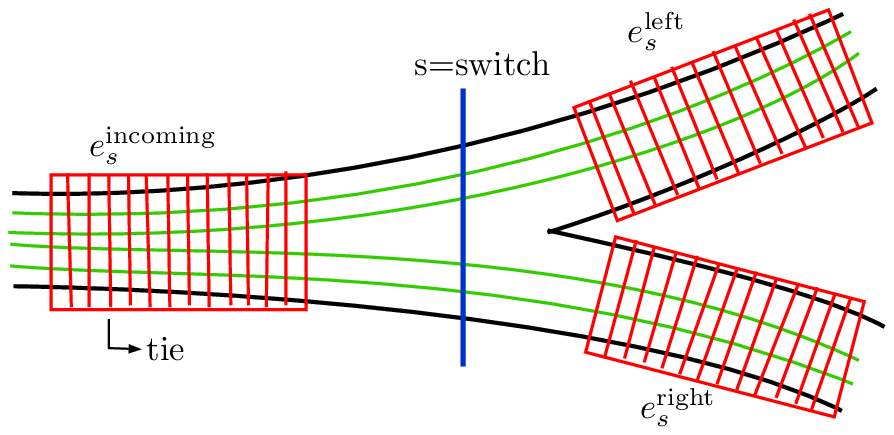}
\end{center}
\end{figure}


The space ${\mathcal H}(\lambda,\mathbb{R})$ has also anti-symmetric
bilinear form, namely the Thurston symplectic form
$\omega_{\mathrm{Thurston}}.$ Let $\lambda$ be a  maximal geodesic lamination on $\Sigma$ and 
$\Phi$ be a fattened train-track carrying the maximal geodesic
lamination.

The Thurston symplectic form is the anti-symmetric bilinear form $\omega_{\mathrm{Thurston}}:\mathcal{H}(\lambda;\mathbb{R})\times
\mathcal{H}(\lambda;\mathbb{R})\to \mathbb{R}$ defined by
$$\omega_{\mathrm{Thurston}}(\sigma_1,\sigma_2)=\frac{1}{2}\sum_{s}
\text{det} \left[\begin{array}{cc}
   \sigma_1(e_s^{\text{left}}) & \sigma_1(e_s^{\text{right}}) \\
   \sigma_2(e_s^{\text{left}}) & \sigma_2(e_s^{\text{right}})
   \end{array}\right],
$$
 where  $\sigma_i(e)\in \mathbb{R}$ is the weight associated to
the edge $e$ by the transverse cocycle $\sigma_i.$ Note that, $\omega_{\mathrm{Thurston}}$ is actually independent of
the train-track $\Phi.$

It is proved in \cite{Soz-Bon} that up to a multiplicative constant,
$\omega_{\mathrm{Thurston}}$ is the same as
$\omega_{\mathrm{PSL(2,\mathbb{R})}}$, and hence is in the same
equivalence class of $\omega_{\mathrm{WP}}.$ More precisely, for
the identification $\text{T}_{\varrho}\mathrm{Teich}(\Sigma)\cong
{\mathcal H}(\lambda;\mathbb{R}),$ the following is valid
$\omega_{\mathrm{PSL(2,\mathbb{R})}}=2\omega_{\mathrm{Thurston}}.$

As a final word on this study, Reidemeister torsion of $\varrho\in
\mathrm{Teich}(\Sigma)$ can be expressed in terms of
$\omega_{\mathrm{Thurston}}.$




\bibliographystyle{aipproc}   



\end{document}